\documentclass[a4paper]{amsart}
\usepackage[cp1251]{inputenc}
\usepackage[russian,english]{babel}

\PII{ISSN 0351--336X}
\makeatletter
\def\@setcopyright{\@empty}
\makeatother

\newcommand{\prn}[1]{\left(#1\right)}

\newcommand{\allp}{1\le p\le\infty}
\newcommand{\Lp}{L_{p,\alpha}}

\newcommand{\norm}[1]{\left\|#1\right\|_{p,\alpha}}
\newcommand{\normpar}[2]{\left\|#1\right\|_{#2}}
\newcommand{\E}{E_n(f)_{p,\alpha}}
\newcommand{\Epar}[2]{E_{#1}\left(#2\right)_{p,\alpha}}
\newcommand{\T}[3]{T_{#1}^{#2}\left(#3\right)}
\newcommand{\hatT}[3]{\hat T_{#1}^{#2}\left(#3\right)}
\newcommand{\Si}[1]{\left(1-#1^2\right)}

\newcommand{\w}{\hat\omega(f,\delta)_{p,\alpha}}
\newcommand{\wpar}[2]{\hat\omega_{#1}\left(#2\right)_{p,\alpha}}
\newcommand{\Px}[1]{P_{#1}^{(2,2)}}
\newcommand{\Py}[1]{P_{#1+2}^{(0,0)}}

\newtheorem{thm}{Theorem}[section]
\newtheorem{lmm}{Lemma}[section]

\newcounter{const}

\numberwithin{const}{thm}
\numberwithin{const}{lmm}
\numberwithin{const}{cor}

\makeatletter
\newcommand{\Cn}[1][]{%
  \stepcounter{const}C_{\theconst}%
  \@ifnotempty{#1}{\newcounter{#1}\setcounter{#1}{\arabic{const}}}}
\makeatother

\newcommand{\lastC}{C_{\theconst}}
\newcommand{\prevC}[1][1]{%
	{\countdef\n=255
	 \n=\theconst
	 \advance\n by-#1
	 C_{\number\n}}}

\numberwithin{equation}{section}

\renewcommand{\theconst}{\arabic{const}}

\DeclareMathOperator*\esssup{ess\ sup}

\begin{document}

\title[A note on an inverse theorem\dots]
	{A note on an inverse theorem for a generalised
		modulus of smoothness%
	}
\author{Muharrem Q.~Berisha}
\author{Faton M.~Berisha}
\address{F.~M.\ Berisha\\
	Faculty of Mathematics and Sciences\\
	University of Prishtina\\
	N\"ena Terez\"e~5\\
	10000 Prishtin\"e\\
	Kosov\"e%
}
\email{faton.berisha@uni-pr.edu}

\date{}

\begin{abstract}
	We prove the theorem converse to Jackson's theorem
	for a modulus of smoothness of the first order
	generalised by means of an asymmetric operator
	of generalised translation.
\end{abstract}

\maketitle

\section*{Introduction}

The relation between the modulus of smoothness
and best approximation by trigonometric polynomials
of a $2\pi$-periodic function is well-known.
In the case of non periodic functions
there is no such relation between their moduli of smoothness
and best approximation by algebraic polynomials.
An analogy with the $2\pi$-periodic case takes place
if the ordinary modulus of smoothness
is replaced by a generalised modulus of smoothness
(see e.g.\
	\cite{butzer-s-w:approx-79,ditzian-t:moduli,
		potapov:vestnik-83%
	}).

In number of papers generalised moduli of smoothness
are introduced by means of generalised symmetric
operators of translation
\cite{potapov:trudy-75,potapov:trudy-81,potapov:vestnik-83}.

In~\cite{potapov:mat-99},
an asymmetric operator of generalised translation is introduced,
by means of it
a generalised modulus of smoothness of the first order is defined,
and the theorem of coincidence
of the class of functions defined by that modulus
with the class of functions
with given order of best approximation by algebraic polynomials
is proved.

In the present paper
we prove a theorem converse to Jackson's theorem
related to that modulus of smoothness.

\section{Definitions}

By~$L_p$ we denote the set of functions~$f$
measurable on the segment~$[-1,1]$
such that for $1\le p<\infty$
\begin{displaymath}
	\normpar f p
	=\prn{\int_{-1}^1|f(x)|^p\,dx}^{1/p}<\infty,
\end{displaymath}
and for $p=\infty$
\begin{displaymath}
	\normpar f\infty
	=\esssup_{-1\le x\le1}|f(x)|<\infty.
\end{displaymath}

Denote by~$\Lp$ the set of functions~$f$
such that
$f(x)\*(1-x^2)^\alpha\in L_p$,
and put
\begin{displaymath}
	\norm f=\normpar{f(x)(1-x^2)^\alpha}p.
\end{displaymath}

By~$\E$ we denote best approximation of a function $f\in\Lp$
by algebraic polynomials of degree not greater than~$n-1$,
in~$\Lp$ metrics,
i.e.
\begin{displaymath}
	\E=\inf_{P_n}\norm{f-P_n},
\end{displaymath}
where $P_n$ are algebraic polynomials
of degree not greater than~$n-1$.

For a function~$f$
we define an operator of generalised translation $\hatT t{}{f,x}$
by
\begin{multline*}
	\hatT t{}{f,x}
	=\frac1{\pi\Si x}\int_0^\pi
		\bigg(
			1-\prn{x\cos t-\sqrt{1-x^2}\sin t\cos\varphi}^2\\
			-2\sin^2t\sin^2\varphi+4\Si{x}\sin^2t\sin^4\varphi
		\bigg)\\
		\times f(x\cos t-\sqrt{1-x^2}\sin t\cos\varphi)\,d\varphi.
\end{multline*}

By means of that operator of generalised translation
we define the generalised modulus of smoothness by
\begin{displaymath}
	\w=\sup_{|t|\le\delta}\norm{\hatT t{}{f,x}-f(x)}.
\end{displaymath}

Put $y=\cos t$, $z=\cos\varphi$
in the operator $\hatT t{}{f,x}$,
we denote it by $\T y{}{f,x}$
and rewrite it in the form
\begin{multline*}
	\T y{}{f,x}
	  =\frac1{\pi\Si x}\int_{-1}^1
		  \big(
			1-R^2-2\Si y\Si z\\
	+4\Si x\Si y\Si{z}^2
		  \big)
		  f(R)\frac{dz}{\sqrt{1-z^2}},
\end{multline*}
where $R=xy-z\sqrt{1-x^2}\sqrt{1-y^2}$.

By $P_\nu^{(\alpha,\beta)}(x)$ $(\nu=0,1,\dotsc)$
we denote the Jacobi's polynomials,
i.e.\ algebraic polynomials of degree~$\nu$ orthogonal
with the weight function $(1-x)^{\alpha}(1+x)^{\beta}$
on the segment $[-1,1]$
and normed by the condition
$P_\nu^{(\alpha,\beta)}(1)=1$ $(\nu=0,1,\dotsc)$.

Denote by $a_n(f)$ the Fourier--Jacobi coefficients
of a function~$f$,
integrable with the weight function $\Si{x}^2$
on the segment $[-1,1]$,
with respect to the system of Jacobi polynomials
$\{\Px n(x)\}_{n=0}^\infty$,
i.e.
\begin{displaymath}
	a_n(f)=\int_{-1}^1f(x)\Px n(x)\Si{x}^2\,dx
	\quad(n=0,1,\dotsc).
\end{displaymath}

The following properties of the operator $T_y$
are proved in~\cite{potapov:mat-99}.

\begin{lmm}\label{lm:properties-T}
	Operator~$T_y$ has the following properties
	\begin{enumerate}
	\item
		The operator $\T y{}{f,x}$ is linear
		with respect to~$f$;
	\item
		$\T1{}{f,x}=f(x)$;
	\item
		$\T y{}{\Px n,x}=\Px n(x)R_n(y)$
		$(n=0,1,\dotsc)$,\\
		where $R_n(y)=\Py n(y)+\frac32\Si y\Px n(y)$;
	\item
		$\T y{}{1,x}=1$;
	\item
		$a_k\prn{\T y{}{f,x}}=R_k(y)a_k(f)$
		$(k=0,1,\dotsc)$.
	\end{enumerate}
\end{lmm}

\begin{lmm}\label{lmm:th:coincidence}
	Let given numbers~$p$, $\alpha$, $r$ and~$\lambda$
	be such that $\allp$, $0<\lambda<2$;
	\begin{alignat*}2
		\frac12      &<\alpha\le1
		  &\quad &\text{for $p=1$},\\
		1-\frac1{2p} &<\alpha<\frac32-\frac1{2p}
		  &\quad &\text{for $1<p<\infty$},\\
		1            &\le\alpha<\frac32
		  &\quad &\text{for $p=\infty$}.
	\end{alignat*}
	Let $f\in\Lp$.
	Then
	\begin{displaymath}
		\E\le\frac{\Cn}{n^\lambda}
	\end{displaymath}
	if and only if
	\begin{displaymath}
		\w\le\Cn\delta^\lambda,
	\end{displaymath}
	where constants~$\prevC$ and~$\lastC$
	do not depend on~$f$, $n$ and~$\delta$.
\end{lmm}

The lemma is proved in~\cite{potapov:mat-99}.

\section{The converse theorem}

Now we formulate our result.

\begin{thm}\label{th:converse}
	Let given numbers~$p$, $\alpha$ and~$\lambda$
	be such that $\allp$, $0<\lambda<2$;
	\begin{alignat*}2
		\frac12      &<\alpha\le1
		  &\quad &\text{for $p=1$},\\
		1-\frac1{2p} &<\alpha<\frac32-\frac1{2p}
		  &\quad &\text{for $1<p<\infty$},\\
		1            &\le\alpha<\frac32
		  &\quad &\text{for $p=\infty$}.
	\end{alignat*}
	If $f\in\Lp$,
	then the following inequality holds
	\begin{displaymath}
		\wpar{}{f,\frac1n}
		\le\frac\Cn{n^2}\sum_{\nu=1}^n\nu\Epar\nu f,
	\end{displaymath}
	where the constant~$C$ does not depend on~$f$ and~$n$.
\end{thm}

\begin{proof}
	Let $P_n(x)$ be the polynomial
	of degree not greater than $n-1$
	such that
	\begin{displaymath}
		\norm{f-P_n}=\E \quad(n=1,2,\dotsc),
	\end{displaymath}
	and
	\begin{displaymath}
		Q_k(x)=P_{2^k}(x)-P_{2^{k-1}}(x) \quad(k=1,2,\dotsc),
	\end{displaymath}
	$Q_0(x)=P_1(x)$.
	
	For given~$n$
	we chose the positive integer~$N$
	such that
	\begin{displaymath}
		\frac n2<2^N\le n+1.
	\end{displaymath}
	By the proof of Lemma~\ref{lmm:th:coincidence}
	given in~\cite{potapov:mat-99}
	it follows that
	\begin{multline*}
		\wpar{}{f,\frac1n}
		\le\Cn
			\bigg(
			  \Epar{2^N}f
			  +\frac1{n^2}\sum_{\mu=1}^N 2^{2\mu}\norm{Q_k}
			\bigg)\\
		\le2\lastC
			\bigg(
				\Epar{2^N}f
				+\frac1{n^2}
					\sum_{\mu=1}^N2^{2\mu}
						\prn{\Epar{2^\mu}f+\Epar{2^{\mu-1}}f}
			  \bigg)\\
		\le4\lastC
			\bigg(
				\Epar{2^N}f
				+\frac1{n^2}
					\sum_{\mu=0}^{N-1}2^{2(\mu+1)}\Epar{2^\mu}f
			\bigg)\\
		\le\frac\Cn{n^2}
			\sum_{\mu=0}^N2^{2(\mu+1)}\Epar{2^\mu}f.
	\end{multline*}
	Considering that for $\mu\ge1$
	we have
	\begin{displaymath}
		\sum_{\nu=2^{\mu-1}}^{2^\mu-1}\nu\Epar\nu f
		\ge\Epar{2^\mu}f 2^{2(\mu-1)},
	\end{displaymath}
	it follows that
	\begin{multline*}
		\wpar{}{f,\frac1n}
		\le\frac\Cn{n^2}
			\bigg(
				2^2\Epar1f
				+\sum_{\mu=1}^N\sum_{\nu=2^{\mu-1}}^{2^\mu-1}
					\nu\Epar\nu f
			\bigg)\\
		\le\frac\Cn{n^2}\sum_{\nu=1}^n\nu\Epar\nu f.
	\end{multline*}
	
	Theorem~\ref{th:converse} is proved.
\end{proof}

\bibliographystyle{amsplain}
\bibliography{maths}

\end{document}